\documentclass[a4paper,11pt]{amsart}
\usepackage[english]{babel}
\usepackage{amssymb}
\usepackage{amscd}
\usepackage{mathtools}
\usepackage[paper=a4paper,left=3cm,right=3cm,top=2cm,bottom=2.0cm]{geometry}
\usepackage{enumerate}
\usepackage{hyperref}
\usepackage{tikz}
\usepackage{tikz-cd}

\newtheorem{thm}{Theorem}[section]
\newtheorem*{thm-non}{Theorem}

\newtheorem{lem}[thm]{Lemma}
\newtheorem{prop}[thm]{Proposition}
\newtheorem{cor}[thm]{Corollary}
\theoremstyle{definition}
\newtheorem{defi}[thm]{Definition}
\newtheorem{rem}[thm]{Remark}

\DeclareMathOperator\Br{Br}
\DeclareMathOperator\OO{\mathcal{O}}
\DeclareMathOperator\End{End}
\DeclareMathOperator\Hom{Hom}
\DeclareMathOperator\Ext{Ext}
\DeclareMathOperator\Coh{Coh}
\DeclareMathOperator\id{id}
\DeclareMathOperator\M{M}
\DeclareMathOperator\Hilb{Hilb}
\DeclareMathOperator\Fix{Fix}

\DeclareMathOperator\Ima{Im}

\DeclareMathOperator\NS{NS}
\DeclareMathOperator\Pic{Pic}

\newcommand\Mv{\M_{\oX,\oh}(v)}
\newcommand\tE{\Theta(\overline{E})}
\newcommand\tEa{\Theta(\overline{E_1})}
\newcommand\tEb{\Theta(\overline{E_2})}
\newcommand\oX{\overline{X}}
\newcommand\oD{\overline{D}}
\newcommand\op{\overline{p}}
\newcommand\oq{\overline{q}}
\newcommand\oh{\overline{h}}
\newcommand\oW{\overline{W}}
\newcommand\oE{\overline{E}}
\newcommand\oY{\overline{Y}}
\newcommand\oEa{\overline{E_1}}
\newcommand\oEb{\overline{E_2}}

\DeclareMathOperator\mA{\mathcal{A}}
\DeclareMathOperator\omA{\overline{\mathcal{A}}}

\begin{document}

\title[Enriques surfaces with trivial Brauer map]{Enriques surfaces with trivial Brauer map and involutions on hyperk\"ahler manifolds}

\author{Fabian Reede}
\address{Institut f\"ur Algebraische Geometrie, Leibniz Universit\"at Hannover, Welfengarten 1, 30167 Hannover, Germany}
\email{reede@math.uni-hannover.de}


\begin{abstract}
Let $X$ be an Enriques surface. Using Beauville's result about the triviality of the Brauer map of $X$, we define a new involution on the category of coherent sheaves on the canonically covering K3 surface $\overline{X}$. We relate the fixed locus of this involution to certain Picard schemes of the noncommutative pair $(X,\mA)$, where $\mA$ is an Azumaya algebra on $X$ defined by the nontrivial element in the Brauer group of $X$.
\end{abstract}

\maketitle

\section*{Introduction}
Let $X$ be an Enriques surface. The universal cover $\oX$ of $X$ is known to be a K3 surface. The covering $q: \oX \rightarrow X$ is an \'{e}tale double cover with covering involution $\iota$. 

The universal cover induces a map between Brauer groups, the so-called Brauer map of $X$: $q^{*}: \Br(X) \rightarrow \Br(\oX)$. Since $\Br(X)\cong \mathbb{Z}/2\mathbb{Z}$ it is a natural question to determine whether the Brauer map is trivial. Beauville answers this question completely in \cite{beau}: the Brauer map of an Enriques surface $X$ is trivial, if and only if $\oX$ admits a line bundle $L=\OO_{\oX}(\ell)$ which is anti-invariant with respect to $\iota$, that is $\iota^{*}L=L^{-1}$, and such that $\ell^2 \equiv 2 \pmod{4}$.

The nontrivial element in $\Br(X)$ can be represented by an Azumaya algebra $\mA$ of rank four on $X$, a quaternion algebra. The triviality of the Brauer map implies that the pullback $\omA$ to $\oX$ is a trivial Azumaya algebra of the form $\mathcal{E}nd_{\oX}(F)$. In the first section we give an explicit description of such a locally free sheaf $F$ of rank two. Then the functor 
\begin{equation*}
	\Theta: \Coh_l(\oX,\omA) \rightarrow \Coh(\oX),\,\,\, G \mapsto F^{*}\otimes_{\omA} G
\end{equation*}
is a Morita equivalence. Here $\Coh_l(\oX,\omA)$ is the the category of coherent sheaves on $\oX$ which are also left $\omA$-modules and $F^{*}$ is seen as a right $\omA$-module.

Using Beauville's result we define and study the following involution:
\begin{equation*}
	\sigma: \Coh(\oX)\rightarrow \Coh(\oX),\,\,\,G \mapsto \sigma(G):=\iota^{*}G\otimes L.
\end{equation*}
One observation is that we have the following relation: $\Theta\circ\iota^{*}=\sigma\circ\Theta$. This shows that if a coherent left $\mA$-module $G$ is fixed by $\iota^{*}$ then $\Theta(G)$ is fixed by $\sigma$. 

The main result in the second section states that a torsion free sheaf $G$ of rank two on $\oX$, which is fixed by $\sigma$, is slope semistable with respect to a polarization of the from $\oh$, where $h$ is a polarization on $X$. By standard results about polarizations and walls, we find that such sheaves are in fact stable for certain choices of Mukai vectors $v$. We study their moduli spaces $\Mv$ and show that $\sigma$ restricts to an anti-symplectic involution of $\Mv$ and thus gives rise to a Lagrangian subscheme $L$ given by $\Fix(\sigma)$.

In the third section we study moduli spaces $\mathrm{M}_{\mA/X}(v_{\mA})$ that classify coherent torsion free sheaves on $X$ that are also left $\mA$-modules, such that they are generically of rank one over the division ring $\mA_{\eta}$. These spaces where constructed by Hoffmann and Stuhler in \cite{hoff}. We prove that such an $\mA$-module $E$ defines a smooth point if $\Theta(\oE)$ is slope stable on $\oX$. We show that in theses cases $\mathrm{M}_{\mA/X}(v_{\mA})$ is an \'{e}tale double cover of $\Fix(\sigma)$ and that the locus of locally projective $\mA$-modules is dense in $\mathrm{M}_{\mA/X}(v_{\mA})$.

In the last section we consider the case that $\mathrm{M}_{\mA/X}(v_{\mA})$ is singular. We give an explicit description of the structure of $\Theta(\oE)$ if $E$ defines a singular point. We end by showing that $\mathrm{M}_{\mA/X}(v_{\mA})$ is generically smooth by adapting a result of Kim in \cite{kim} to this situation.

In this article we consider Enriques surfaces over the complex numbers $\mathbb{C}$ with trivial Brauer map such that $\rho(\oX)=11$. This is the first case where a trivial Brauer map is possible.

\section{Enriques surfaces with trivial Brauer map}\label{trivbr}
Let $X$ be an Enriques surface, that is $\mathrm{H}^1(X,\OO_X)=0$ and $\omega_X\neq \OO_X$ is 2-torsion.
We have the canonical \'{e}tale double cover $q:\oX \rightarrow X$ induced by $\omega_X$. It is well known that $\oX$ is a K3 surface. Denote the covering involution by $\iota: \oX \rightarrow \oX$. 

It is also well known that 
\begin{equation*}
\Br(X)\cong \mathbb{Z}/2\mathbb{Z}=\left\langle \alpha \right\rangle\,\,\text{and}\,\, \Br(\oX)\cong\Hom(\mathrm{T}_{\oX},\mathbb{Q}/\mathbb{Z})
\end{equation*}
where $\mathrm{T}_{\oX}$ is the transcendental lattice of $\oX$, see \cite[Corollary 5.7.1]{dolg} and \cite[Section 2]{beau}

The canonical cover induces a map on Brauer groups, the so called Brauer map:
\begin{equation*}
q^{*}:  \Br(X) \rightarrow \Br(\oX).
\end{equation*}
In \cite[Proposition 3.4, Corollary 5.7]{beau} Beauville gives an explicit description of the element $q^{*}\alpha$ as well as the following equivalence for the triviality of the Brauer map:

\begin{thm}\label{beau1}
Let $X$ be an Enriques surface. The Brauer map $q^{*}: \Br(X)\rightarrow \Br(\oX)$ is trivial if and only if there is $L=\OO_{\oX}(\ell)\in \Pic(\oX)$ with $\iota^{*}L=L^{-1}$ and $\ell^2 \equiv 2 \pmod{4}$.
\end{thm}

The lattice $q^{*}\NS(X)$ is a primitive rank 10 sublattice in $\NS(\oX)$, that is we must have $\rho(\oX)\geq 10$. This sublattice is in fact the invariant part of the action of the induced involution $\iota^{*}$ on $\NS(\oX)$. 

More exactly (see e.g. \cite[Theorem 5.1]{hori}): there is an involution $\tau$ on the K3-lattice $\Lambda_{\mathrm{K3}}$ decomposing the lattice as $\Lambda_{\mathrm{K3}}=\Lambda^{+}\oplus\Lambda^{-}$ according to the eigenspaces of $\tau$. Now it is possible to choose a marking $\varphi: \mathrm{H}^2(\oX,\mathbb{Z}) \xrightarrow{\,\,\cong\,\,}\Lambda_{\mathrm{K3}}$ such that $\tau\circ \varphi = \varphi\circ \iota^{*}$. Then by \cite[Proposition 2.3]{namika} one has 
\begin{equation*}
\Lambda^{+}\cap \NS(\oX)=q^{*}\NS(X).
\end{equation*}

If $X$ is a very general Enriques surface then \cite[Proposition 5.6]{namika} gives the equality
\begin{equation*}
\NS(\oX)=q^{*}\NS(X)\cong\NS(X)(2)\,\,\,\text{resp.}\,\,\,\NS(\oX)\cap \Lambda^{-}=0,
\end{equation*}
i.e. there are no $\iota^{*}$-anti-invariant line bundles. Hence in these cases the Brauer map is non-trivial. So the first interesting case happens possibly for Enriques surfaces with $\rho(\oX)=11$.

In \cite{ohashi} Ohashi classified all K3 surfaces with $\rho=11$ allowing for a fixed point free involution, that is K3 surfaces that cover an Enriques surface. And indeed by \cite[Proposition 3.5]{ohashi} there are K3 surfaces with Enriques quotient $q:\oX \rightarrow X$ satisfying
\begin{equation*}
\NS(\oX)=q^{*}\NS(X) \oplus \mathbb{Z}L\,\,\,\text{with}\,\,\,L=\OO_{\oX}(\ell)\,\,\text{such that}\,\, \ell^2=-2N,\,\,\,N\geqslant 2
\end{equation*}
and by the decomposition of the K3-lattice we see
\begin{equation*}
\Lambda^{-}\cap \NS(\oX)=\mathbb{Z}L\,\,\,\text{i.e.}\,\,\,\iota^{*}L=L^{-1}.
\end{equation*}
Thus if we choose an odd $N\geqslant 3$, we see that there are Enriques surfaces $X$ with associated K3 surface satisfying $\rho(\oX)=11$ such that all conditions of Theorem \ref{beau1} are satisfied. We fix such an Enriques surface $X$ in the following.

\begin{defi}
The autoequivalence $\sigma_{(\iota,L)}$ of $\Coh(\oX)$ associated to the pair $(\iota,L)$ is defined to be
\begin{equation*}
\sigma_{(\iota,L)} : \Coh(\oX) \rightarrow \Coh(\oX),\,\,\, G \mapsto \iota^{*}G\otimes L.
\end{equation*}
Since $\iota^{*}$ is an involution and $L$ is $\iota^{*}$-anti-invariant, we see that in fact $\sigma_{(\iota,L)}$ is also an involution. In the following we denote this involution simply by $\sigma$. 
\end{defi}

\begin{rem}\label{lincoh}
The line bundle $L$ defines a non-zero element in the group cohomology $\mathrm{H}^1(G,\Pic(\oX))$ for $G=\left\langle \iota^{*} \right\rangle$. More exactly $L$ is in the kernel of $\id\otimes \iota^{*}$ but not in the image of $\id\otimes \iota^{*}(-)^{-1}$, see \cite[Corollary 4.3]{beau}.
\end{rem}

In \cite[Proposition 3.3]{reede} we proved that the Brauer class $\alpha$ can be represented by a quaternion algebra $\mathcal{A}$ on $X$. Denote by $p: Y \rightarrow X$ the Brauer-Severi variety associated to $\mathcal{A}$. This is a $\mathbb{P}^1$-bundle which is not of the form $\mathbb{P}(E)$ for any locally free $\mathcal{O}_X$-module $E$ of rank 2. Since $q^{*}\alpha=0$ in $\Br(\oX)$ it is known that $\omA=q^{*}\mA\cong \mathcal{E}nd_{\oX}(F)$ for some locally free sheaf $F$ of rank two on $\oX$.

To find a candidate for $F$ we note that in \cite[Lemma 10]{mart} Mart\'{i}nez defines $E:=\mathcal{O}_{\oX} \oplus L$ and shows that $\mathbb{P}(E)\rightarrow \oX$ descends to a $\mathbb{P}^1$-bundle over $X$, which does not come from a locally free sheaf. This $\mathbb{P}^1$-bundle therefore must agree with the Brauer-Severi variety $Y\rightarrow X$ associated to $\mA$ and have Brauer class $\alpha$. 

By \cite[8.4]{quill}, we get the following cartesian diagram
\begin{equation*}
	\begin{tikzcd}
		\mathbb{P}(E) \arrow[r,"\oq"]\arrow[d,swap,"\op"] & Y \arrow[d,"p"]\\
		\oX \arrow[r,"q"] & X
	\end{tikzcd}
\end{equation*}	
together with an isomorphism $\omA:=q^{*}\mathcal{A}\cong \mathcal{E}nd_{\oX}(E)$. 

\begin{rem}
Quillen actually considers the opposite algebra $\mA^{op}$. We can ignore the opposite algebra, as $\mA$ has order two in the Brauer group, that is, there is an isomorphism $\mA\cong \mA^{op}$. In general using the opposite algebra is a \emph{convention}, depending on the question if the Brauer-Severi variety of $\mA$ classifies certain right or left ideals, see \cite[Warning 24]{kol}. 
\end{rem}

To have nicer formulas in the following, we will use $\det(E)=L$ and the isomorphism
\begin{equation*}
E^{*}\cong E\otimes \det(E)^{-1} = E\otimes L^{-1}.
\end{equation*}
Defining $F:=E^{*}$, the isomorphism gives rise to induced isomorphisms
\begin{equation*}
\omA\cong \mathcal{E}nd_{\oX}(E)\cong \mathcal{E}nd_{\oX}(E\otimes L^{-1})\cong \mathcal{E}nd_{\oX}(E^{*})= \mathcal{E}nd_{\oX}(F). 
\end{equation*}

Recall that $F$ is a left $\mathcal{E}nd_{\oX}(F)$-module and $F^{*}$ is a right one. In this situation we have the following form of Morita equivalence between the category of coherent left $\omA$-modules and coherent $\OO_{\oX}$-modules, see \cite[Proposition 8.26]{wed}:
\begin{align*}
\Theta: &\Coh_l(\oX,\omA) \xrightarrow{\sim} \Coh(\oX),\,\,\, H \mapsto F^{*}\otimes_{\omA}H\\
\intertext{with inverse is given by}
\Xi: &\Coh(\oX) \xrightarrow{\sim} \Coh_l(\oX,\omA),\,\,\,E \mapsto F\otimes E.
\end{align*}

The next lemma studies the relation between $\Theta$ and the involutions $\iota^{*}$ and $\sigma$.
\begin{lem}\label{moritaiota}
For $G\in\Coh_l(\oX,\omA)$ there is an isomorphism 
\begin{equation*}
\Theta(\iota^{*}G) \cong \sigma(\Theta(G)).
\end{equation*}
\end{lem}
\begin{proof}
We first note that indeed $\iota^{*}G\in \Coh(\oX,\omA)$ as $\iota^{*}\omA\cong \omA$, that is Morita equivalence for $\iota^{*}G$ is well defined. Further we have an isomorphism
\begin{equation*}\label{F-iso}
\iota^{*}F=\iota^{*}\left( \mathcal{O}_{\oX}\oplus L^{-1}\right)\cong \mathcal{O}_{\oX}\oplus L\cong (\OO_X \oplus L^{-1})\otimes L\cong F \otimes L. 
\end{equation*}
Using this isomorphism as well as $G\cong \Xi(\Theta(G)) \cong F\otimes\Theta(G)$ we find
\begin{align*}
\iota^{*}G &\cong \iota^{*}(F\otimes \Theta(G)) \cong \iota^{*}F\otimes \iota^{*}\Theta(G)\cong F\otimes L\otimes\iota^{*}\Theta(G) \\
&\cong F\otimes (\iota^{*}\Theta(G)\otimes L)\cong F\otimes \sigma(\Theta(G))\cong \Xi(\sigma(\Theta(G))).
\end{align*}
Applying $\Theta(-)$ once more gives the desired isomorphism.
\end{proof}

The following corollary contains an easy but crucial observation:

\begin{cor}\label{fix}
Assume $G\in\Coh_l(\oX,\omA)$ is fixed by $\iota^{*}$, then $\Theta(G)\in \Fix(\sigma)$.
\end{cor}


\begin{rem}
The corollary applies especially to those $G\in \Coh_l(\oX,\omA)$ which are in the image of $q^{*}:\Coh_l(X,\mA)\rightarrow \Coh_l(\oX,\omA)$.
\end{rem}

\section{Stable sheaves and involutions on hyperk\"ahler manifolds}
The last section suggests to study sheaves on $\oX$ which are fixed under the involution $\sigma$. We first start with their numerical data:
\begin{lem}
Let $G$ be a coherent torsion free $\OO_{\oX}$-module with rank $r$. If $G\in \Fix(\sigma)$ then $r=2a$ for some $a\in \mathbb{N}$ and $\operatorname{c}_1(G)=\oD+a\ell$ for some $D\in \NS(X)$. 
\end{lem}
\begin{proof}
Write $\operatorname{c}_1(G)=\oD+a\ell$ for some $D\in \NS(X)$ and $a\in \mathbb{Z}$. Since $G$ is fixed under $\sigma$ we find
\begin{equation*}
\oD+a\ell=\operatorname{c}_1(G)=\operatorname{c}_1(\sigma(G))=\operatorname{c}_1(\iota^{*}G\otimes L)=\iota^{*}(\oD+a\ell)+r\ell=\oD+(r-a)\ell
\end{equation*}
Since $\NS(\oX)$ is torsion free, this implies $r=2a$. 
\end{proof}

\begin{cor}
Let $G$ be a coherent torsion free $\OO_{\oX}$-module. If $G\in \Fix(\sigma)$ then the Mukai vector has the form
\begin{equation*}
v(G)=(2s,\oD+s\ell,\chi(G)-2s)=v(\sigma(G))
\end{equation*}
for some $D\in \NS(X)$ and some $s\in \mathbb{N}$
\end{cor}

Next we want to study slope-(semi)stability of sheaves which are fixed under the involution $\sigma$. For this we recall that for any polarization $h\in \NS(X)$ we have that $\oh\in \NS(\oX)$ is a polarization on $\oX$, since $q$ is finite. It thus makes sense to study $\mu_{\oh}$-(semi)stability of $G\in \Fix(\sigma)$. We will do this for the first non-trivial case, that is with Mukai vector
\begin{equation*}
v(G)=(2,\oD+\ell,\chi(G)-2).
\end{equation*}

We need the following result, which holds more generally, but this will suffices for us:
\begin{lem}\label{subsheaf}
	Let $E$ be a torsion free sheaf on $\oX$ and assume $F_1$ and $F_2$ are saturated rank one subsheaves of $E$. Then either one has $F_1\cap F_2=0$ or $F_1= F_2$.
\end{lem}
\begin{proof}
Let $T_i$ denote the torsion free quotient of $E$ by $F_i$. We have two induced morphisms $\alpha_1: F_1 \rightarrow T_2$ and $\alpha_2: F_2 \rightarrow T_1$ with kernel $F_1\cap F_2$. 

If one of the morphisms is nontrivial it must be injective as both sheaves are torsion free and the $F_i$ are of rank one. But this implies it has trivial kernel and thus $F_1\cap F_2=0$.

So assume both morphisms are zero. Then we get $F_1 \subseteq F_2\subseteq F_1$ and thus $F_1= F_2$.	
\end{proof}

The following theorem is based on \cite[Lemma 3.5, Proposition 3.6]{chan}:

\begin{thm}\label{semist}
Let $G$ be a coherent torsion free $\OO_{\oX}$-module of rank two with $G\in \Fix(\sigma)$, then $G$ is $\mu_{\oh}$-semistable for any polarization $h$ on $X$.
\end{thm}
\begin{proof}
Since $\ell$ is $\iota^{*}$-anti-invariant and $\oh$ is $\iota^{*}$-invariant we find
\begin{equation*}
\operatorname{c}_1(L)\oh=\ell\oh=0.
\end{equation*}
This implies for a torsion free sheaf $M$ of rank $r$:
\begin{equation}\label{conesigma}
\operatorname{c}_1(\sigma(M))\oh=\operatorname{c}_1(\iota^{*}M\otimes L)\oh=(\iota^{*}\operatorname{c}_1(M)+r\operatorname{c}_1(L))\oh=\operatorname{c}_1(M)\oh.
\end{equation}

To check semistability, it is enough to consider saturated rank one subsheaves, as $G$ has rank two. Let $N \hookrightarrow G$ be such subsheaf. Since $G$ is fixed under the involution $\sigma$ we find that $\sigma(N)\hookrightarrow G$ is also a saturated subsheaf of rank one.

It is impossible to have $N=\sigma(N)$ as subsheaves of $G$. Indeed this would imply that we have $\det(N) = \det(\sigma(N))$. But then
\begin{equation*}
\det(N) = \det(\sigma(N))\Leftrightarrow \det(N)\cong\iota^{*}\det(N)\otimes L \Leftrightarrow \det(N)\otimes \left(\iota^{*}\det(N) \right)^{-1}\cong L 
\end{equation*}
so that $L$ would be in image of $\id\otimes\left( \iota^{*}(-)\right)^{-1}$, which it is not by Remark \ref{lincoh}. 

So by Lemma \ref{subsheaf} we have $N\cap \sigma(N)=0$. Therefore there is an injection $N\oplus \sigma(N)\hookrightarrow G$. 

We compute slopes using \eqref{conesigma}:
\begin{equation*}
\mu_{\oh}(N\oplus\sigma(N))=\frac{\operatorname{c}_1(N\oplus\sigma(N))\oh}{2}=\operatorname{c}_1(N)\oh=\mu_{\oh}(N).
\end{equation*}
Since $N\oplus \sigma(N)$ is a rank two subsheaf of $G$ we also have 
\begin{equation*}
	\mu_{\oh}(N\oplus\sigma(N))\leqslant \mu_{\oh}(G),
\end{equation*}
see for example \cite[Lemma 4.3]{fried}. We conclude $\mu_{\oh}(N)\leqslant \mu_{\oh}(G)$ and $G$ is $\mu_{\oh}$-semistable.
\end{proof}

One may wonder if there are cases in which $G$, or more generally all semistable sheaves with the same numerical invariants as $G$, are in fact $\mu_{\oh}$-stable. To answer this question we start with the following lemma:

\begin{lem}\label{walls}
Let $h\in \NS(X)$ be any polarization on $X$, then $\oh\in \NS(\oX)$ is not on a wall of type $(2,\Delta)$ with $0< \Delta < -\ell^2$.
\end{lem}
\begin{proof}
Recall (see \cite[Definition 4.C.1]{huy}) that a class $\xi\in \NS(\oX)$ is of type $(r,\Delta)$ if we have $-\frac{r^2}{4}\Delta \leqslant \xi^2 < 0$ and the wall $W_{\xi}$ of type $(r,\Delta)$
defined by $\xi$ is 
\begin{equation*}
W_{\xi}:=\left\lbrace [H]\in \mathcal{H}\,|\, \xi H=0 \right\rbrace.
\end{equation*}

Assume $\oh$ is on a wall of type $(2,\Delta)$. We have $\xi \oh=0$ for a class $\xi$ with $-\Delta\leqslant\xi^2 < 0$. Write $\xi=\overline{D}+a\ell$ for some $D\in \NS(X)$ and $a\in \mathbb{Z}$ then
\begin{equation*}
\xi \oh=0 \Leftrightarrow \oD \oh=0.
\end{equation*}
Using the Hodge Index theorem we find $\oD^2\leqslant 0$. It follows that 
\begin{equation*}
\xi^2 = (\overline{D}+a\ell)^2 = \oD^2+a^2\ell^2\leqslant \ell^2.
\end{equation*}
Thus if we have $\ell^2 < -\Delta < 0$ then $-\Delta\leqslant\xi^2<-\Delta$, a contradiction. Hence $\oh$ is not on a wall $W_{\xi}$ of type $(2,\Delta)$.

%
%
%
\end{proof}

%

We are now able to prove the $\mu_{\oh}$-stability of $G$ in some cases:

\begin{thm}\label{stable}
Let $G$ be a coherent torsion free $\mu_{\oh}$-semistable $\OO_{\oX}$-module. If $G$ has Mukai vector $v(G)=(2,\oD+\ell,\chi(G)-2)$ such that $0<v(G)^2+8<-\ell^2$, then $G$ is $\mu_{\oh}$-stable for any polarization $h\in \NS(X)$. 
\end{thm}

\begin{proof}
We check that all conditions of \cite[Theorem 4.C.3]{huy} are satisfied: as $\oX$ is a K3 surface we have $\NS(\oX)=\mathrm{Num}(\oX)$. The class $\operatorname{c}_1(G)=\oD+\ell$ is indivisible in $\NS(\oX)$ as $\ell$ is primitive and the summand $\oD$ comes from the orthogonal complement of $\ell$ in $\NS(\oX)$. 

A quick computation shows that the discriminant of $G$ is given by
\begin{equation*}
\Delta(G)=v(G)^2+8.
\end{equation*}
By Lemma \ref{walls} the polarization $\oh$ is not on a wall of type $(2,\Delta(G))$ for any polarization $h$ on $X$. It follows that every $\mu_{\oh}$-semistable sheaf with the given numerical invariants is actually $\mu_{\oh}$-stable.
\end{proof}

Denote the Mukai vector $v(G)=(2,\oD+\ell,\chi(G)-2)$ of $G$ simply by $v$ and let $\Mv$ be the moduli space of $\mu_{\oh}$-semistable sheaves on $\oX$ with Mukai vector $v$. If $0 < v^2+8 < -\ell^2$ then by Theorem \ref{stable} every $\mu_{\oh}$-semistable sheaf in $\Mv$ is $\mu_{\oh}$-stable. Thus in this case any polarization of the form $\oh$ is $v$-generic. 

As the first Chern class is indivisible by a well known result $\Mv$ is an irreducible holomorphic symplectic variety, deformation equivalent to $\Hilb^n(\oX)$ with $2n=v^2+2$, particularly $\Mv\neq \emptyset$. In the following we assume that we are in this situation. 

The involution $\sigma$ certainly preserves $\mu_{\oh}$-stability, that is if $G$ is $\mu_{\oh}$-stable, then so is $\sigma(G)=\iota^{*}G\otimes L$. This follows as $\iota^{*}G$ is slope-stable with respect to $\iota^{*}\oh=\oh$ and the tensor product with a line bundle does not affect stability. As $v$ is the Mukai vector of $G\in \Fix(\sigma)$ we have $v(\sigma(G))=v$ so that in fact the involution $\sigma$ restricts to an involution
\begin{equation*}
\sigma: \Mv \rightarrow \Mv,\,\,\,G\mapsto \sigma(G)=\iota^{*}G\otimes L.
\end{equation*}

Recall Mukai's construction of a holomorphic symplectic form on $\Mv$ using the Yoneda- (or cup-) product and the trace map, see \cite{muk} for more details:
\begin{equation*}
\Ext^1_{\oX}(G,G)\times \Ext^1_{\oX}(G,G) \xrightarrow{\cup} \Ext^2_{\oX}(G,G)\xrightarrow{\mathrm{tr}} \mathrm{H}^2(\oX,\OO_{\oX})\cong \mathbb{C}.
\end{equation*}
We see that there are the following isomorphisms for $i\geqslant 0$:
\begin{equation*}
\Ext^i_{\oX}(\sigma(G),\sigma(G))=\Ext^i_{\oX}(\iota^{*}G\otimes L,\iota^{*}G\otimes L)\cong \Ext^i_{\oX}(\iota^{*}G,\iota^{*}G).
\end{equation*} 
But $\iota^{*}$ is known to be antisymplectic with respect to Mukai's form, so $\sigma$ is also an antisymplectic involution. By a result of Beauville, see \cite[Lemma 1]{beau2}, it follows that $\Fix(\sigma)\subset \Mv$ is a smooth Lagrangian subscheme of dimension $n$ if it is not empty.

\begin{prop}
The fixed locus $\Fix(\sigma)$ in $\Mv$ is not empty.
\end{prop}
\begin{proof}
We have $v=(2,\oD+\ell,\chi(G)-2)$. A computation shows
\begin{equation*}
v^2=(\oD+\ell)^2-4(\chi(G)-2)=\oD^2+\ell^2-4(\chi(G)-2)\equiv 2\pmod{4}
\end{equation*}
which follows from $\oD^2\equiv 0\pmod{4}$ and $\ell^2\equiv 2\pmod{4}$. Thus we have
\begin{equation*}
v^2+2\equiv 0\pmod{4}.
\end{equation*}

It is also well known that if $Y$ is a hyperk\"ahler manifold of dimension $2r$ then we have $\chi(\OO_Y)=r+1$. Thus in our case $\chi(\OO_{\Mv})=2k+1$ for some $k\in\mathbb{N}$.

Now if $\sigma$ were fixed point free it would induce an \'{e}tale double cover 
\begin{equation*}
\Mv \rightarrow \Mv/\left\langle \sigma \right\rangle.
\end{equation*}
But this would imply that $\chi(\OO_{\Mv})$ is even, a contradiction. So $\sigma$ must have fixed points. 
\end{proof}


\section{Twisted Picard schemes: smooth cases}
Let $X$ still be an Enriques surface with trivial Brauer map $q:\Br(X)\rightarrow \Br(\oX)$ as described in Section \ref{trivbr}. Denote the quaternion algebra representing the nontrivial element $\alpha\in \Br(X)$ by $\mathcal{A}$. As seen before, one has $\omA\cong \mathcal{E}nd_{\oX}(F)$. In this section we want to study Picard schemes of the noncommutative version $(X,\mA)$ of the classical pair $(X,\OO_X)$.

\begin{defi}
A sheaf $E$ on $X$ is called a generically simple torsion free $\mathcal{A}$-module if 
\begin{enumerate}
\item $E$ is coherent and torsion free as a $\mathcal{O}_X$-module and
\item $E$ is a left $\mathcal{A}$-module such that the generic stalk $E_{\eta}$ is a simple module over the $\mathbb{C}(X)$-algebra $\mathcal{A}_{\eta}$.
\end{enumerate}
Since in our case $\mathcal{A}_{\eta}$ is a division ring over $\mathbb{C}(X)$, $E$ is also called a torsion free $\mathcal{A}$-module of rank one.
\end{defi}

Choosing a polarization $h$ on $X$, Hoffmann and Stuhler showed that these modules are classified by a moduli space, more exactly we have (see \cite[Theorem 2.4. iii), iv)]{hoff}):

\begin{thm}
There is a projective moduli scheme $\M_{\mathcal{A}/X;\operatorname{c}_1,\operatorname{c}_2}$ classifying torsion free $\mathcal{A}$-modules of rank one with Chern classes $\operatorname{c}_1\in \NS(X)$ and $\operatorname{c}_2\in \mathbb{Z}$.
\end{thm} 

\begin{rem}
The moduli scheme $\M_{\mathcal{A}/X;\operatorname{c}_1,\operatorname{c}_2}$ can be thought of as a noncommutative Picard scheme $\Pic_{\operatorname{c}_1,\operatorname{c}_2}(\mA)$ for the pair $(X,\mA)$.
\end{rem}

In \cite{reede} we studied $\M_{\mathcal{A}/X;\operatorname{c}_1,\operatorname{c}_2}$ for an Enriques surface with nontrivial Brauer map by pulling everything back to $\oX$. This cannot work in this case as the pullback $\oE$ of a torsion free $\mathcal{A}$-module $E$ of rank one to $\oX$ is not a generically simple $\omA$-module anymore.

But using Morita equivalence we see that given a torsion free $\mathcal{A}$-module of rank one on $X$, we have $\oE\cong F\otimes \Theta(\oE)$ for the pullback $\oE$ on $\oX$. 

\begin{defi}
Let $S$ be an arbitrary smooth projective surface. Given an Azumaya algebra $\mathcal{B}$ on $S$ one we define the $\mathcal{B}$-Mukai vector for an $\mathcal{B}$-module $E$ by
\begin{equation*}
v_{\mathcal{B}}(E):=\operatorname{ch}(E)\sqrt{\operatorname{td}(S)}{\sqrt{\operatorname{ch}(\mathcal{B})}}^{-1}.
\end{equation*}
As in the case of $\mathcal{O}_S$-modules, it has the property that
\begin{equation*}
v_{\mathcal{B}}(E)^2=-\chi_{\mathcal{B}}(E,E)=\sum\limits_{i=0}^2 (-1)^{i+1} \dim_{\mathbb{C}}\left( \mathrm{E}xt^{i}_{\mathcal{B}}(E,E)\right).
\end{equation*}
\end{defi}

Instead of studying the moduli space $\M_{\mathcal{A}/X;\operatorname{c}_1,\operatorname{c}_2}$ we will consider the moduli space $\M_{\mathcal{A}/X}(v_{\mA})$ of torsion free $\mA$-modules of rank one with $\mA$-Mukai vector $v_{\mA}$ in the following.

By \cite[Proposition 3.5.]{hoff} we have the following form of Serre duality in this case:

\begin{prop}\label{serre}
Let $E_1$ and $E_2$ be coherent left $\mathcal{A}$-modules. There are the following isomorphisms for $0\leqslant i\leqslant 2$:
\begin{equation*}
\Ext^i_{\mathcal{A}}(E_1,E_2)\cong \Ext^{2-i}_{\mathcal{A}}(E_2,E_1\otimes\omega_{X})^{*}.
\end{equation*}
\end{prop}

\begin{lem}\label{exts}
Let $E_1$ and $E_2$ be coherent left $\mA$-modules. There are the following isomorphisms for $0\leqslant i\leqslant 2$:
\begin{equation*}
\begin{aligned}
\Ext^i_{\omA}(\overline{E_1},\overline{E_2})&\cong \Ext^i_{\oX}(\Theta(\overline{E_1}),\Theta(\overline{E_2}))\\
\Ext^i_{\omA}(\overline{E_1},\overline{E_2})&\cong \Ext^i_{\mA}(E_1,E_2)\oplus \Ext^i_{\mA}(E_1,E_2\otimes\omega_X).
\end{aligned}
\end{equation*}  
\end{lem}

\begin{proof}
The first isomorphism is simply Morita equivalence. For the second isomorphism, we note that all classical relations between the various functors on $\OO_X$- and $\OO_{\oX}$-modules are also valid in the noncommutative case of $\mA$- and $\omA$-modules, see \cite[Appendix D]{kuz}. Especially we have isomorphisms
\begin{equation*}
\Ext^i_{\omA}(\overline{E_1},\overline{E_2})\cong \Ext^i_{\mA}(E_1,q_{*}q^*E_2)\,\,\,\,(0\leqslant i\leqslant 2).
\end{equation*}
Applying the projection formula for finite morphisms together with $q_{*}\OO_{\oX}\cong \OO_X\oplus \omega_X$ finally gives the second isomorphism.
\end{proof}

\begin{cor}\label{mukvec}
Let $E$ be a coherent left $\mA$-module, then 
\begin{equation*}
v(\Theta(\oE))^2=2v_{\mA}(E)^2
\end{equation*}
\end{cor}

\begin{proof}
We have the following equalities:
\begin{align*}
v(\Theta(\oE))^2&=-\chi_{\oX}(\Theta(\oE),\Theta(\oE))=-\chi_{\omA}(\oE,\oE)\\
&=-\chi_{\mA}(E,E)-\chi_{\mA}(E,E\otimes\omega_X)=-2\chi_{\mA}(E,E)=2v_{\mA}(E)^2
\end{align*}
Here the second and third equality is Lemma \ref{exts}. The fourth equality is Serre duality for $\mA$-modules, see Proposition \ref{serre}.
\end{proof}

\begin{thm}
Let $E$ be a torsion free $\mA$-module of rank one, then $\Theta(\oE)$ is $\mu_{\oh}$-semistable. If $0 < 2v_{\mA}(E)^2+8 < -\ell^2$ then $\Theta(\oE)$ is $\mu_{\oh}$-stable.
\end{thm}

\begin{proof}
Since $E$ is a torsion free $\mA$-module of rank one, it has rank four as an $\OO_{\oX}$-module, so $\Theta(\oE)$ has rank two. Now Lemma \ref{fix} shows that $\Theta(\oE)\in \Fix(\sigma)$ so it is $\mu_{\oh}$-semistable by Theorem \ref{semist}. Using Corollary \ref{mukvec} we have 
\begin{equation*}
0 < 2v_{\mA}(E)^2+8 < -\ell^2\,\, \Leftrightarrow\,\, 0 <v(\Theta(\oE))^2+8 < -\ell^2
\end{equation*}
which shows that $\Theta(\oE)$ is $\mu_{\oh}$-stable by Theorem \ref{stable}.
\end{proof}

The theorem shows that for certain numerical invariants we have a morphism
\begin{equation*}
\phi: \M_{\mathcal{A}/X}(v_{\mA}) \rightarrow \Mv,\,\,[E]\mapsto \left[ \Theta(\oE)\right] .
\end{equation*}

We already saw that $\Ima(\phi)\subset \Fix(\sigma)$ and that in this case the fixed locus is never empty. In fact we also have the reverse inclusion

\begin{lem}\label{nonemp}
Assume $0 < 2v_{\mA}^2+8 < -\ell^2$. Then $\M_{\mathcal{A}/X}(v_{\mA})$ is nonempty if and only if $\Fix(\sigma)$ is nonempty. Furthermore we have $\Fix(\sigma)\subset \Ima(\phi)$.
\end{lem}

\begin{proof}
As mentioned before if $[E]\in \M_{\mathcal{A}/X}(v_{\mA})$ then $\left[ \Theta(\oE)\right] \in \Fix(\sigma)\subset \Mv$.

So take $[G]\in \Fix(\sigma)\subset \Mv$. Then we have
\begin{equation*}
\sigma(G)\cong G\,\,\,\Leftrightarrow\,\,\,\iota^{*}G\cong G\otimes L^{-1}.
\end{equation*}
Define $H:=\Xi(G)=F\otimes G$. This is a left $\omA$-module and satisfies
\begin{equation*}
\End_{\omA}(H)\cong \End_{\oX}(G)\cong \mathbb{C},
\end{equation*}
using Morita equivalence and the simplicity of $G$ (as it is $\mu_{\oh}$-stable by our assumptions).

Furthermore we have the following isomorphism of $\omA$-modules:
\begin{equation*}
\iota^{*}H\cong  \iota^{*}F\otimes \iota^{*}G \cong  \left(F\otimes L \right)\otimes \left(G\otimes L^{-1} \right)\cong H.
\end{equation*}
By \cite[Theorem 2.6]{reede} we have $H\cong \oE$ for some torsion free $\mA$-module $E$ of rank one on $X$, so $\Theta(\oE)=G$, that is $[G]\in \Ima(\phi)$ and $\M_{\mathcal{A}/X}(v_{\mA})$ is not empty. 
\end{proof}

\begin{thm}
Assume $0 < 2v_{\mA}^2+8 < -\ell^2$. Then 
\begin{enumerate}[i)]
\item $\M_{\mathcal{A}/X}(v_{\mA})$ is smooth and an \'{e}tale double cover of $\Fix(\sigma)$.
\item The locus of locally projective $\mathcal{A}$-modules of rank one is dense in $\M_{\mathcal{A}/X}(v_{\mA})$.
\end{enumerate}
\end{thm}

\begin{proof}
The obstruction to smoothness of $\M_{\mathcal{A}/X}(v_{\mA})$ at a point $[E]$ lies in $\Ext^2_{\mA}(E,E)$, which is Serre dual to $\Hom_{\mA}(E,E\otimes\omega_X)^{*}$. Now by the stability of $\tE$ and Lemma \ref{exts} there are isomorphisms
\begin{equation*}
\mathbb{C}\cong\End_{\oX}(\tE)\cong \End_{\omA}(\oE)\cong \End_{\mA}(E)\oplus\Hom_{\mA}(E,E\otimes \omega_X).
\end{equation*}
As $E$ is a simple $\mA$-module, we have $\End_{\mA}(E)\cong\mathbb{C}$ so that $\Hom_{\mA}(E,E\otimes\omega_X)=0$. Therefore all obstructions vanish and the moduli space is smooth.

We have already seen that 
\begin{equation*}
\phi: \M_{\mathcal{A}/X}(v_{\mA}) \rightarrow \Mv,\,\,\, [E] \rightarrow \left[\tE \right]
\end{equation*}
factors through $\Fix(\sigma)$ and in fact by Lemma \ref{nonemp} we have $\Ima(\phi)=\Fix(\sigma)$. Thus $\phi$ induces a surjective morphism
\begin{equation*}
\varphi: \M_{\mathcal{A}/X}(v_{\mA}) \rightarrow \Fix(\sigma)
\end{equation*}
betweens smooth schemes.

Assume $\varphi(\left[ E_1\right] )=\varphi([E_2])$. That is we have an isomorphism $\tEa\cong \tEb$ and thus 
\begin{equation*}
\oEa \cong \Xi(\tEa)\cong \Xi(\tEb)\cong \oEb.
\end{equation*}
So we must have $E_1\cong E_2$ or $E_1\cong E_2\otimes\omega_X$ but not both as
\begin{equation*}
\mathbb{C}\cong \Hom_{\oX}(\tEa,\tEb)\cong \Hom_{\omA}(\oEa,\oEb)\cong \Hom_{\mA}(E_1,E_2)\oplus\Hom_{\mA}(E_1,E_2\otimes \omega_X).
\end{equation*}
It follows that the morphism $\varphi$ is unramified and $2:1$. By \cite[Lemma]{schaps} it is also flat, hence \'{e}tale.

To see that the locus of locally projective $\mathcal{A}$-modules is dense, similar to \cite[Theorem 4.10 (ii)]{reede}, it is enough to prove that $\Ext^2_{\mA}(E^{**},E)=0$. This vanishing implies that the connecting homomorphism
	\begin{equation*}
		\begin{tikzcd}
			\cdots \arrow[r] &\Ext_{\mA}^1(E,E) \arrow[r,"\delta"] & \Ext_{\mA}^2(T,E) \arrow[r] & \Ext_{\mA}^2(E^{**},E) \arrow[r] & \cdots
		\end{tikzcd}
	\end{equation*}	
of the long exact sequence we get after applying $\Hom_{\mA}(-,E)$ to the bidual sequence
	\begin{equation*}\label{bidual}
		\begin{tikzcd}
			0 \arrow[r] & E \arrow[r] & E^{**} \arrow[r] & T \arrow[r] & 0
		\end{tikzcd}
	\end{equation*}	
is surjective, which then allows to use the rest of the proof of \cite[Theorem 3.6. iii)]{hoff}. But $\Ext^2_{\mA}(E^{**},E)$ is Serre dual to $\Hom_{\mA}(E,E^{**}\otimes\omega_X)^{*}$. To prove the vanishing of the latter, we claim that there is an isomorphism 
\begin{equation*}
\Theta(\overline{E^{**}})\cong\Theta(\oE)^{**}.
\end{equation*}
Indeed we have following isomorphisms:
\begin{equation*}
F\otimes\Theta(\overline{E^{**}}) \cong \overline{E^{**}} \cong \oE^{**} \cong \left(F\otimes\Theta(\oE) \right)^{**}\cong F\otimes \Theta(\oE)^{**}.
\end{equation*}
Here the first isomorphism is Morita equivalence for $\overline{E^{**}}$, the second isomorphism is flatness of $q:\oX \rightarrow X$, the third is Morita equivalence for $\oE$ and the final isomorphism uses the locally freeness of $F$.

This isomorphism shows that $\Theta(\overline{E^{**}})$ is $\mu_{\oh}$-stable since $\Theta(\oE)$ is. Especially $\Theta(\overline{E^{**}})$ is simple as an $\OO_{\oX}$-module and hence so is $\overline{E^{**}}$ as an $\omA$-module. It follows from \cite[Lemma 1.7.]{reede} that we have $\Hom_{\mA}(E,E^{**}\otimes\omega_X)=0$.
\end{proof}

\section{Twisted Picard schemes: singular cases}
In this section we want to study the case that $[E]\in \M_{\mathcal{A}/X}(v_{\mA})$ is a singular point. This implies that 
\begin{equation*}
\Ext^2_{\mA}(E,E)\cong \Hom_{\mA}(E,E\otimes\omega_X)\cong \mathbb{C}.
\end{equation*} 
Especially there is an isomorphism of $\mA$-modules
\begin{equation*}
E\cong E\otimes \omega_X.
\end{equation*}
To study the structure of such $\mA$-modules we first prove a more general statement. For this we need some notation: let $W$ be a smooth projective variety together with an \'{e}tale Galois double cover $q: \overline{W}\rightarrow W$ with covering involution $\iota$. The Brauer-Severi variety of an Azumaya algebra $\mathcal{A}$ on $W$ is denoted by $p: Y\rightarrow W$. We get the following diagram with cartesian squares 
\begin{equation}\label{diag}
\begin{tikzcd}
\overline{Y} \arrow{r}{\overline{\iota}}\arrow{d}[swap]{\overline{p}} & \overline{Y} \arrow{r}{\overline{q}}\arrow{d}[swap]{\overline{p}} & Y \arrow{d}{p}\\
\overline{W} \arrow{r}{\iota} & \overline{W} \arrow{r}{q} & W
\end{tikzcd}
\end{equation}  
Here $\overline{q}: \overline{Y}\rightarrow Y$ is also an \'{e}tale Galois double cover with covering involution $\overline{\iota}$. Again, by \cite[8.4]{quill}, we have
\begin{equation*}
\mA_Y:=p^{*}\mathcal{A}\cong \mathcal{E}nd_Y(G)\,\,(\text{and thus}\,\, \mA\cong p_{*}\mathcal{E}nd_Y(G))
\end{equation*}
for a locally free sheaf $G$ on $Y$ which is compatible with base change and if $Y=\mathbb{P}(E)$, i.e. $\mA=\mathcal{E}nd_W(E)$, we have $G=p^{*}E\otimes \OO_Y(-1)$.


Then we have the following equivalences
\begin{align*}
\phi: \Coh_l(W,\mathcal{A}) \rightarrow \Coh(Y,W),\,\,\, & E\mapsto G^{*}\otimes_{\mathcal{A}_Y}p^{*}E\\
\psi: \Coh(Y,W) \rightarrow \Coh_l(W,\mathcal{A}),\,\,\, & E\mapsto p_{*}(G\otimes E) 
\end{align*}
with
\begin{equation*}
\Coh(Y,W)=\left\lbrace E\in \Coh(Y)\,|\, p^{*}p_{*}(G\otimes E)\xrightarrow{\,\cong\,} G\otimes E \right\rbrace.
\end{equation*}
We have similar equivalences $\overline{\phi}$ and $\overline{\psi}$ involving $\overline{\mathcal{A}}_{\overline{Y}}\cong \mathcal{E}nd_{\overline{Y}}(\overline{q}^{*}G)$, $\overline{Y}$ and $\overline{W}$.

\begin{rem}\label{eqmor}
If $\mA=\mathcal{E}nd_W(E)$ is trivial, i.e. $Y=\mathbb{P}(E)$, we can compose the equivalences $\phi$ and $\psi$ with Morita equivalence and get the following equivalences, using the isomorphism $G\cong p^{*}E\otimes\OO_Y(-1)$:
\begin{align*}
\Coh(W) \rightarrow \Coh(Y,W),\,\,\, & H\mapsto p^{*}H\otimes \OO_Y(1)\\
\Coh(Y,W) \rightarrow \Coh(W),\,\,\, & H\mapsto p_{*}(H\otimes \OO_Y(-1)) 
\end{align*}
with
\begin{equation*}
\Coh(Y,W)=\left\lbrace H\in \Coh(Y)\,|\, p^{*}p_{*}(H\otimes \OO_Y(-1))\xrightarrow{\,\cong\,} H\otimes \OO_Y(-1) \right\rbrace.
\end{equation*}
\end{rem}

\begin{lem}\label{push}
If for $M\in \Coh(Y,W)$ there is $N\in \Coh(\overline{Y})$ such that $M\cong \overline{q}_{*}N$ then $N\in \Coh(\overline{Y},\overline{W})$
\end{lem}

\begin{proof}
We have to prove that the canonical morphism
\begin{equation*}
\phi: \op^{*}\op_{*}(\oq^{*}G\otimes N)\rightarrow \oq^{*}G\otimes N
\end{equation*}
is an isomorphism. But the morphism $\oq: \overline{Y}\rightarrow Y$ is finite which implies that the (underived) direct image functor $\oq_{*}$ is conservative, that is we have
\begin{equation*}
\phi\,\,\text{is an isomorphism}\,\, \Leftrightarrow\,\, \oq_{*}(\phi)\,\,\text{is an isomorphism}.
\end{equation*}
Using the flatness of $\op$ and \cite[Proposition 12.6]{wed}, diagram \ref{diag} and the projection formula, we find the following chain of isomorphisms
\begin{alignat*}{2}
&\oq_{*}\op^{*}\op_{*}(\oq^{*}G\otimes N)&&\rightarrow \oq_{*}\left( \oq^{*}G\otimes N\right) \\
\cong\,\,\,\, & p^{*}q_{*}\op_{*}(\oq^{*}G\otimes N)&&\rightarrow \oq_{*}\left( \oq^{*}G\otimes N\right) \\
\cong\,\,\,\, & p^{*}p_{*}\oq_{*}(\oq^{*}G\otimes N)&&\rightarrow \oq_{*}\left( \oq^{*}G\otimes N\right) \\
\cong\,\,\,\, & p^{*}p_{*}(\left( G\otimes \oq_{*}N\right))&&\rightarrow G\otimes \left( \oq_{*}N\right)  \\
\cong\,\,\,\, & p^{*}p_{*}(G \otimes M)&&\rightarrow G \otimes M.
\end{alignat*}
But $M\in \Coh(Y,W)$, so the last morphism is an isomorphism. But then so is the first, which is $\oq_{*}(\phi)$ and hence also $\phi$. Thus $N\in \Coh(\overline{Y},\overline{W})$.
\end{proof}

Now we return to our special situation. That is $W=X$ is an Enriques surface with trivial Brauer map as in Section \ref{trivbr}, $\oW=\oX$ the covering K3 surface, $Y$ is the Brauer-Severi variety of the Azumaya algebra $\mA$ corresponding to the nontrivial class $\alpha\in \Br(X)$. By the triviality of the Brauer map we have $\omA=\mathcal{E}nd_{\oX}(F)$ and therefore $\oY\cong\mathbb{P}(F)$.
\begin{lem}\label{isomline}
There is an isomorphism of line bundles
\begin{equation*}
\overline{\iota}^{*}\OO_{\oY}(1)\cong \OO_{\oY}(1)\otimes \op^{*}L
\end{equation*}
\end{lem}
\begin{proof}
Note that the induced involution $\overline{\iota}: \oY \rightarrow \oY$ actually factorizes in the following way, using the isomorphism $\oY\cong \mathbb{P}(F)$:
\begin{equation*}\label{diag2}
\begin{tikzcd}
\oY\cong\mathbb{P}(F)\arrow[bend left=25]{rr}{\overline{\iota}}\arrow{r}{\,\,\beta} & \mathbb{P}(F\otimes L)\cong\mathbb{P}(\iota^{*}F)\arrow{r}{\alpha} & \mathbb{P}(F)\cong\oY
\end{tikzcd}
\end{equation*}  
Here $\alpha: \mathbb{P}(\iota^{*}F) \rightarrow \mathbb{P}(F)$ is induced by the base change along the involution $\iota: \oX \rightarrow \oX$, which by \cite[Remark 13.27]{wed} implies
\begin{equation*}
\alpha^{*}\OO_{\oY}(1)=\OO_{\mathbb{P}(\iota^{*}F)}(1)
\end{equation*}
Furthermore as $\mathbb{P}(\iota^{*}F)\cong\mathbb{P}(F\otimes L)$ the map $\beta: \mathbb{P}(F)\rightarrow \mathbb{P}(F\otimes L)$ is the canonical $\oX$-isomorphism described in \cite[Remark 13.35]{wed} with 
\begin{equation*}
\beta^{*}\OO_{\mathbb{P}(F\otimes L)}(1)\cong \OO_{\mathbb{P}(F)}(1)\otimes \op^{*}L\cong \OO_{\oY}(1)\otimes \op^{*}L.
\end{equation*}
Putting both facts together gives the desired isomorphism of line bundles.
\end{proof}

\begin{lem}\label{descrip}
Let $E$ be a coherent left $\mA$-module such that there is an isomorphism of $\mA$-modules $E\cong E\otimes \omega_{X}$. Then there is a coherent sheaf $B$ on  $\oX$ such that $\Theta(\oE)\cong B\oplus\sigma(B)$.
\end{lem}

\begin{proof}
Assume $E\cong E\otimes \omega_X$ as left $\mA$-modules. Using the equivalence $\phi$, there is an induced isomorphism on $Y$:
\begin{equation*}
\phi(E)\cong \phi(E)\otimes p^{*}\omega_X.
\end{equation*}
We must have $\phi(E)\cong \oq_{*}C$ for some $C\in \Coh(\oY)$ as $p^{*}\omega_X$ defines the double cover $\oY\rightarrow Y$. By Lemma \ref{push} we have $C\in \Coh(\oY,\oX)$ and thus $E\cong p_{*}(G\otimes\oq_{*}C)$. We find
\begin{equation*}
\oE = q^{*}E\cong q^{*}p_{*}(G\otimes\oq_{*}C)\cong \op_{*}\oq^{*}\left(G\otimes\oq_{*}C \right)\cong \op_{*}\left(\oq^{*}G \otimes\oq^{*}\oq_{*}C\right).
\end{equation*}

Since $\oq: \oY \rightarrow Y$ is a an \'{e}tale double cover with involution $\overline{\iota}$ we have 
\begin{equation*}
\oq^{*}\oq_{*}C \cong C\oplus \overline{\iota}^{*}C.
\end{equation*}
In addition there is $B\in \Coh(\oX)$ with $C\cong \op^{*}B\otimes \OO_{\oY}(1)$, as explained in Remark \ref{eqmor}. 

Putting all these facts together with $\oq^{*}G\cong \op^{*}F\otimes\OO_{\oY}(-1)$ leads to:
\begin{align*}
\oE &\cong\op_{*}\left(\op^{*}F\otimes \OO_{\oY}(-1) \otimes \left(\op^{*}B\otimes \OO_{\oY}(1)\oplus \overline{\iota}^{*}\left(\op^{*}B\otimes \OO_{\oY}(1) \right) \right) \right)\\
&\cong \op_{*}\left(\op^{*}F \otimes \left(\op^{*}B\oplus \left(\op^{*}\iota^{*}B\otimes \iota^{*}\left( \OO_{\oY}(1)\right)  \right)\otimes \OO_{\oY}(-1) \right) \right).
\end{align*}
Using Lemma \ref{isomline} and the projection formula then show that in fact we have
\begin{align*}
\oE &\cong\op_{*}\left( \op^{*}F \otimes \left(\op^{*}B\oplus \left(\op^{*}\iota^{*}B\otimes \OO_{\oY}(1)\otimes\op^{*}L\right) \otimes \OO_{\oY}(-1) \right) \right) \\
&\cong \op_{*}\op^{*}\left(F\otimes\left(B\oplus\sigma(B) \right)  \right) \cong F\otimes(B\oplus\sigma(B)).
\end{align*}
Morita equivalence then gives the desired isomorphism $\Theta(\oE)\cong B\oplus \sigma(B)$.
\end{proof}

By standard computations using Mukai vectors, we have:

\begin{lem}
Let $B\in \Coh(\oX)$ be a torsion free sheaf of rank one, then 
\begin{equation*}
v(B\oplus \sigma(B))^2=4v(B)^2-\left( \operatorname{c}_1(B)-\operatorname{c}_1(\sigma(B))\right)^2.
\end{equation*}
\end{lem}

\begin{proof}
By additivity we find
\begin{equation*}
v(B\oplus \sigma(B))^2=\left(v(B)+v(\sigma(B))\right)^2=v(B)^2+2v(B)v(\sigma(B))+v(\sigma(B))^2 
\end{equation*}
Since $B$ and $\sigma(B)$ are of rank one, the squares of their Mukai vectors only depend on $\operatorname{c}_2$. But $\operatorname{c}_2(\sigma(B))=\operatorname{c}_2(B)$, so
\begin{equation*}
v(B\oplus \sigma(B))^2=2v(B)^2+2v(B)v(\sigma(B)).
\end{equation*}
Write $v(B)=(1,D,\frac{1}{2}D^2-\operatorname{c}_2+1)$ so $v(\sigma(B))=(1,\sigma(D),\frac{1}{2}\sigma(D)^2-\operatorname{c}_2+1)$. It follows that
\begin{equation*}
v(\sigma(B))=v(B)+(0,\sigma(D)-D,\tfrac{1}{2}(\sigma(D)^2-D^2)).
\end{equation*}
Finally we have
\begin{align*}
v(B)v(\sigma(B))&=v(B)^2+v(B)(0,\sigma(D)-D,\tfrac{1}{2}(\sigma(D)^2-D^2))\\
&=v(B)^2+D\sigma(D)-D^2-\tfrac{1}{2}(\sigma(D)^2-D^2)\\
&=v(B)^2-\tfrac{1}{2}(D-\sigma(D))^2.
\end{align*}
Putting all steps together gives the desired result.
\end{proof}

We finish by adapting \cite[\S 2 Theorem (1)]{kim} to our situation:

\begin{thm}
The moduli space $\M_{\mathcal{A}/X}(v_{\mA})$ is singular at $[E]$ if and only if $E\cong E\otimes \omega_X$ as left $\mA$-modules and $[E]$ lies on a component of dimension $v_{\mA}^2+1$. Furthermore one has
\begin{equation*}
\dim(\mathrm{Sing}(\M_{\mathcal{A}/X}(v_{\mA}))) < \frac{1}{2}\left(\dim(\M_{\mathcal{A}/X}(v_{\mA}))+3 \right),
\end{equation*}
that is $\M_{\mA/X}(v_{\mA})$ is generically smooth.
\end{thm}

\begin{proof}
If $[E]$ is a singular point then necessarily the obstruction space does not vanish, hence by Serre duality:
\begin{equation*}
\dim(\Hom_{\mA}(E,E\otimes\omega_X))=\dim(\Ext_{\mA}^2(E,E))>0.
\end{equation*}
The isomorphism $E\cong E\otimes\omega_X$ now follows from \cite[Lemma 4.3]{reede}.
As we have
\begin{equation*}
\dim(T_{[E]}\M_{\mA/X}(v_{\mA}))=\dim(\Ext_{\mA}^1(E,E))=v_{\mA}^2+2,
\end{equation*}
the point $[E]$ must be on a component of dimension $v_{\mA}^2+1$, as it is singular point.

In the other direction, if $E\cong E\otimes\omega_X$ then similarly $\Ext^2_{\mA}(E,E)\cong \mathbb{C}$ and thus
\begin{equation*}
\dim(T_{[E]}\M_{\mA/X}(v_{\mA}))=v_{\mA}^2+2.
\end{equation*}
Since $[E]$ lies on a component of dimension $v_{\mA}^2+1$ the point $[E]$ is singular.

In this situation Lemma \ref{descrip} shows that $\Theta(\oE)\cong B\oplus\sigma(B)$ for a torsion free sheaf of rank one on $\oX$. For $h\in \mathrm{Amp}(X)$ we have $(\operatorname{c}_1(B)-\operatorname{c}_1(\sigma(B))\oh=0$ hence 
\begin{equation*}
(\operatorname{c}_1(B)-\operatorname{c}_1(\sigma(B)))^2 \leqslant 0.
\end{equation*}
In fact an equality never occurs as this is only possibly if $\operatorname{c}_1(B)=\operatorname{c}_1(\sigma(B))$ which cannot happen by Remark \ref{lincoh}.

We have
\begin{align*}
v_{\mA}(E)^2+1 &= \frac{1}{2}v(\Theta(\oE))^2+1=\frac{1}{2}v(B\oplus\sigma(B))^2+1\\
&=2(v(B)^2+2)-3-\frac{1}{2}\left( \operatorname{c}_1(B)-\operatorname{c}_1(\sigma(B))\right)^2\\
&> 2(v(B)^2+2)-3
\end{align*}
Now $\M_{\oX,\oh}(v(B))$ is smooth of dimension $v(B)^2+2$, as it is a Hilbert scheme of points (possibly twisted by a line bundle). Consequently we find
\begin{equation*}
\dim(\M_{\oX,\oh}(v(B)))<\begin{cases}
\frac{1}{2}\left(\dim(\M_{\mA/X}(v_{\mA}))+3\right)\,\,\text{if}\,\, \dim_{[E]}(\M_{\mA/X}(v_{\mA}))=v_{A}^2+1\\
\frac{1}{2}\left(\dim(\M_{\mA/X}(v_{\mA}))+2\right)\,\,\text{if}\,\, \dim_{[E]}(\M_{\mA/X}(v_{\mA}))=v_{A}^2+2
\end{cases}.
\end{equation*}
\end{proof}


\end{document}